\documentclass{amsart}
\usepackage[utf8]{inputenc}
\usepackage[shortlabels]{enumitem}
\usepackage{amssymb}
\usepackage{amsmath}
\usepackage{amsthm}
\usepackage[margin=1in]{geometry}
\usepackage[colorlinks]{hyperref}
\hypersetup{
    colorlinks=true,
    linkcolor=red,
    filecolor=red,      
    urlcolor=blue,
    citecolor=red,
}
\usepackage{theoremref}
\usepackage{tikz-cd}
\usepackage[textwidth=20mm]{todonotes}

\newcommand{\G}{\mathcal{G}}

\newcommand{\R}{\mathbb{R}}

\newcommand{\C}{\mathbb{C}}
\newcommand{\Z}{\mathbb{Z}}
\newcommand{\N}{\mathbb{N}}

\newcommand{\supp}{\text{supp}}

\newcommand{\Gnought}{\G^{(0)}}
\newcommand{\glambda}{\G(\Lambda)}
\newcommand{\chabauty}{\mathcal{C}(G)}
\newcommand{\hull}{\Omega_0(\Lambda)}

\newcommand{\cF}{\mathcal{F}}
\newcommand{\cP}{\mathcal{P}_{\text fin}}

\newcommand{\Linfty}{\mathrm{L}^\infty}

\theoremstyle{plain}
\newtheorem{thm}{Theorem}[section]
\newtheorem{defn}[thm]{Definition}
\newtheorem{example}[thm]{Example}
\newtheorem{lem}[thm]{Lemma}

\newtheorem{cor}[thm]{Corollary}
\newtheorem{prop}[thm]{Proposition}

\newtheorem*{question*}{Question}
\newtheorem{introtheorem}{Theorem}

\newcommand\blfootnote[1]{%
  \begingroup
  \renewcommand\thefootnote{}\footnote{#1}%
  \addtocounter{footnote}{-1}%
  \endgroup
}

\RequirePackage[backend    = biber,
                sortcites  = true,
                giveninits = true,
                doi        = false,
                isbn       = false,
                url        = false,
                maxnames = 50,
                style      = numeric-comp]{biblatex}
\DeclareNameAlias{sortname}{family-given}
\DeclareNameAlias{default}{family-given}
\title{A note on inner amenability for FLC point sets}
\author{Gabriel Favre}
\address{Department of Mathematics,
Stockholm University,
SE-106 91 Stockholm, Sweden.}
\email{favre@math.su.se}

\begin{document}
\begin{abstract}
Inner amenability is a bridge between amenability of an object and amenability of its operator algebras. It is an open problem of Ananantharman-Delaroche to decide whether all étale groupoids are inner amenable. Approximate lattices and their dynamics have recently attracted increased attention and have been studied using groupoid methods. In this note, we prove that groupoids associated with approximate lattices in second countable locally compact groups are inner amenable. In fact we show that this result holds more generally for point sets of finite local complexity in such groups.
\end{abstract}
\maketitle

\blfootnote{
  \textit{MSC classification: 43A07;37A55.}
}
\blfootnote{
  \textit{Keywords: étale groupoids, inner amenability, irregular point sets.}
}

\section{Introduction}
For locally compact groups, inner amenability takes its root in the work of Losert and Rindler \cite{lori87} and was defined by Paterson in \cite{Pa88} as the existence of a Haar continuous conjugation invariant mean. It was shown by Lau-Paterson in \cite{Lapa91} that a locally compact group is amenable if and only if it is both inner amenable and its von Neumann algebra is injective. Losert and Rindler showed in \cite{lori87} that the class of inner amenable groups contains [IN]-groups and that inner amenability is in fact equivalent to having a conjugation invariant neighborhood of the identity for connected groups. More generally, locally compact groups which have a nuclear reduced $C^*$-algebra are inner amenable if and only if they are amenable. This follows from a result of Lau-Paterson (see \cite{Lapa91}). In particular, inner amenable type I groups or almost connected groups are amenable. This follows from Connes in \cite{Co76} for the almost connected case.

It has to be said that a different notion of inner amenability has been introduced for discrete groups by Effros in \cite{ef75} requiring atomlessness of means. It was motivated by its connections\footnote{See also \cite{Va12} for the comparison between property $\Gamma$ and Effros' notion of inner amenability} with property $\Gamma$ of the group's von Neumann algebra. In the present note, we do not require atomlessness of means. It makes discrete groups automatically inner amenable, as the Dirac measure at the identity is then a Haar continuous conjugation invariant mean. 

The natural step of investigating generalizations of inner amenability beyond locally compact groups has been overtaken by Anantharaman-Delaroche. She has defined inner amenability for transformation groupoids in \cite{an00} and more generally for locally compact groupoids in \cite{an21}. Inner amenability of a locally compact groupoid can be roughly defined as the existence of a net of properly supported, continuous, positive type functions on the square of the groupoid which uniformly converges to $1$ on compact subsets of the diagonal. Her notion of inner amenability has been shown to generalize the classical one for locally compact groups by Crann and Tanko in \cite{CrTa17}. As in the group case, inner amenability provides a strong connection between topological amenability of a groupoid and amenability of its operator algebras. Indeed, Anantharaman-Delaroche showed in \cite{an21} that amenable transformation groupoids are precisely inner amenable transformation groupoids whose underlying action is also amenable. Another major result highlighting the role of inner amenability for étale groupoids is the equivalence between topological amenability, amenability at infinity and nuclearity of the reduced $C^*$-algebra for inner amenable étale groupoids. 

The class of inner amenable groupoids has been shown to contain all transformation groupoids arising from discrete group actions. This led Anantharaman-Delaroche to ask the following natural question.
\begin{question*}\cite[Question 11.1 (3)]{an21}
Are all étale groupoids are inner amenable?
\end{question*}
It needs to be emphasized that outside of amenable groupoids and transformation groupoids coming from discrete group actions, very little is known about which étale groupoids are inner amenable. Our main result provides a class of example of inner amenable groupoids which are not constructed as transformation groupoids.

The class of examples of étale groupoids which we study in this note comes from the study of point sets in locally compact second countable groups. In \cite{BjHa18}, Björklund-Hartnick have developed the notion of approximate lattices and initiated their study. Approximate lattices in locally compact groups simultaneously generalize lattices in locally compact groups and Meyer sets in abelian locally compact groups. Analytical properties of approximate lattices and their consequences have recently received attention, see \cite{BjHaPo18,BjHaPo21,BjHaPo22} and \cite{bjha20,bjfi19}.

Approximate lattices are the main motivating example of point sets in locally compact groups which satisfy a regularity condition known as finite local complexity. They provide a natural source of étale groupoids whose construction can be sketched the following way. Given a point set in a second countable locally compact group, consider the action of the group on the Chabauty-Fell closure of the orbit of the point set. Restricting the obtained transformation groupoid to the elements of the Chabauty-Fell closure of the orbit containing the identity of the group yields an étale groupoid when the point set is uniformly discrete. The groupoid obtained through this construction is called the groupoid of the point set. The main result of this paper shows that this groupoid is inner amenable, provided the point set is regular enough.
\begin{introtheorem}\thlabel{thm:intro}
Let $\Lambda\subseteq G$ be a point set of finite local complexity in a locally compact group. Then the groupoid of $\Lambda$ is inner amenable.
\end{introtheorem}
These groupoids appeared first in \cite{BeHeZa00,BoMe19,BoMe21} for point sets in $\R^n$ where the authors use this groupoid description and K-theoretic information to extract information on the underlying dynamical system. The groupoids associated to point sets in general locally compact second countable groups have been defined in greater generality in \cite{enstad20} where they have been used in order to obtain results in frame theory. As a consequence of \thref{thm:intro}, all approximate lattices even in non-amenable groups give rise to inner amenable groupoids.

\section{Point sets and their associated groupoids}
In this section, we recall the necessary notions on point sets in locally compact groups and topological groupoids. The groupoid associated to point sets in locally compact groups can be found in \cite{enstad20}.

\subsection{Point sets and their discrete hull}
In this section, we introduce point sets in locally compact groups and an important dynamical object associated  to each point set: its discrete hull.

Let $G$ be a second countable locally compact group and consider the space $\chabauty$ of all closed subsets of $G$ with the Chabauty-Fell topology given by the following subbasis of topology 
\begin{align*}
O_K&=\{C\in\chabauty\mid C\cap K=\emptyset\} \\
O_U&=\{C\in\chabauty\mid C\cap U\neq \emptyset\},
\end{align*}
where $U$ ranges over all the open subsets of $G$ and $K$ ranges over all the compact subsets of $G$. It is a fact that the space $\chabauty$ is compact and that it is second countable, if $G$ is second countable. The topology of $\chabauty$ is conveniently described in terms the following convergence.
\begin{lem}\cite[Proposition E.1.2]{BePe92}
Let $P_n,P\in\chabauty$ for each $n \in\N$. Then $(P_n)_n$ converges to $P$ if and only if both of the
following statements hold:
\begin{enumerate}
    \item For all $x\in P$ there exists $x_n\in P_n$ such that $(x_n)_n$ converges to $x$;
    \item For all increasing sequences of integers $(n_k)_k$ and $x_{n_k}\in P_{n_k}$ such that $(x_{n_k})_k$ converges to some $x\in G$, we have $x\in P$.
\end{enumerate}
\end{lem} 
\noindent The group $G$ acts continuously on $\chabauty$ by left translation. One defines for $\Lambda \in \chabauty$ the set
$$\Omega(\Lambda):=\overline{\{g\Lambda \mid g\in G\}}\subseteq \chabauty$$
of the closure of all the translates of $\Lambda$ in $\chabauty$ which is called the \textit{hull} of $\Lambda$. As a closed subspace of $\chabauty$, it is compact. Consider the subset $$\hull:=\{P\in\Omega(\Lambda)\mid e\in P\}$$
of elements of the hull containing the identity, called the \textit{discrete hull} of $\Lambda$. It is also compact, since it is closed in $\Omega(\Lambda)$. Before introducing the groupoid associated to point sets, let us recall the necessary groupoid terminology.

\subsection{ The groupoid of a point set}
In this section, we first fix some general notations for topological groupoids and then explain how to associate groupoids to point sets. See \cite[Chapter I]{renault80} for a comprehensive treatment on general groupoid terminology. 

A (topological) groupoid is a topological space $\G$ together with a distinguished set of \textit{units} $\Gnought$, continuous range and source maps $r,s: \G \to \Gnought$ a continuous inversion map $\gamma\to\gamma^{-1}$ from $\G$ to itself and a multiplication map $(\gamma,\eta)\to\gamma\eta$ from $$\G^{(2)}:=\{(\gamma,\eta)\in\G^2\mid s(\gamma)=r(\eta)\}$$ to $\G$ satisfying the following formulas.
\begin{itemize}
    \item $r(\gamma)=\gamma\gamma^{-1}$, $s(\gamma)=\gamma^{-1}\gamma$ for all $\gamma\in\G$;
    \item $(\gamma\eta)\mu=\gamma(\eta\mu)$ for all $(\gamma,\eta),(\eta,\mu)\in\G^{(2)}$;
    \item $(\gamma^{-1})^{-1}=\gamma$ for all $\gamma\in\G$;
    \item $r(\gamma)\gamma=\gamma s(\gamma)=\gamma$ for all $\gamma\in\G$.
\end{itemize}
In this note, groupoids will always be locally compact and Hausdorff. A \textit{bisection} of a groupoid $\G$ is a subset $U \subseteq \G$ such that the restrictions $s|_U$ and $r|_U$ are injective.  A groupoid $\G$ is \emph{{\'e}tale} if its topology has a basis consisting of open bisections.  It is called \emph{ample} if its topology has a basis consisting of compact open bisections. Equivalently, an étale groupoid is ample if (and only if) its unit space is totally disconnected.

If $\G$ is a groupoid and $A \subseteq \Gnought$ is a set of units, one obtains a groupoid by restricting the unit space of $\G$ to $A$ denoted by $\G|_A = \{g \in \G \mid d(g), r(g) \in A\}$. 

We now introduce a groupoid construction associated to a point set in a locally compact group. We use the group action on the hull to get a transformation groupoid, and then restrict its unit space to the discrete hull. A definition and study of groupoids associated with general point sets in locally compact groups can be found in \cite[Section 3]{enstad20}.

For a point set $\Lambda\subseteq G$, the group $G$ acts on the hull $\Omega(\Lambda)$ and one can form a transformation groupoid $G\ltimes \Omega(\Lambda)$. By restricting the unit space of $G\ltimes \Omega(\Lambda)$ to $\hull$, one obtains a groupoid which will be denoted by $\glambda=G\rtimes\Omega(\Lambda)\vert_{\hull}$. The unit space of $\glambda$ is $\hull$. For $\gamma=(x,P)\in\glambda$, the source, range and inverse are given by $s(x,P)=P$, $r(x,P)=xP$ and $(x,P^{-1})=(x^{-1},xP)$. The multiplication in $\glambda$ is given by $(y,xP)(x,P)=(xy,P)$. The groupoid $\glambda$ has the following description which will be repeatedly used in the sequel.
$$\glambda=\{(x,P)\in G\times\hull\mid x^{-1}\in P\}.$$
Note that the topology of $\glambda$ can be described by the convergence of sequences, since $\glambda$ is second countable.

\subsection{Regularity notions for point sets}
In this section, we recall the notions of uniform discreteness and finite local complexity of point sets and describe the impact either property has on the associated groupoid. 

\begin{defn}
A point set $\Lambda$ in a locally compact group $G$ is called \textit{uniformly discrete} if there exists a symmetric open neighborhood $U$ of the identity such that for all $P\in\Omega(\Lambda)$, we have
\begin{equation}\label{eq:uniformly-discrete}
    \vert P \cap U\vert \leq 1.
\end{equation}
If such a symmetric open neighborhood $U$ of the identity exists, we say that $\Lambda$ is $U$-discrete.
\end{defn}
\noindent It is not hard to show that we it is enough to check the condition \eqref{eq:uniformly-discrete} on translates of the point set (see \cite[Proposition 3.3]{enstad20}). For uniformly discrete point sets, there is an explicit basis for the topology of $\glambda$ consisting of open bisections described by the following proposition.
\begin{prop}\thlabel{prop:udiscrete-implies-etale}\cite[Proposition 3.10]{enstad20}
    Let $\Lambda\subseteq G$ be $U_0$-discrete. Let $V$ be a symmetric open neighborhood of the identity with $VV\subset U_0$, $W$ an open subset of $\hull$ and $x\in G$. The set $$U_{x,V,W}=((xV\cap Vx)\times W)\cap\glambda$$ is an open bisection. Further, the collection of all such sets form a basis of the topology of $\glambda$ and in particular, $\glambda$ is an étale groupoid.
\end{prop}
In fact, uniformly discrete point sets are precisely the point sets $\Lambda$ whose associated groupoid is étale (see \cite[Proposition 3.11]{enstad20}). Now we introduce another standard regularity notion. 
\begin{defn}
A point set $\Lambda\subseteq G$ is said to have \textit{finite local complexity} (FLC) if for any compact $K\subseteq G$, there exists a finite collection of finite subsets $\cF\subseteq 2^G$ such that for any $P\in\hull$, we have 
\begin{equation}\label{eq:FLC}
    K\cap P=hF,
\end{equation}for some $h\in G$ and $F\in\cF$.
\end{defn}
It is enough to check the condition \eqref{eq:FLC} on translates of the point set. It is a classical fact that a point set has finite local complexity if and only if $\Lambda^{-1}\Lambda$ is uniformly discrete, which is in turn equivalent to the existence a finite set $F\subseteq G$ such that $\Lambda^{-1}\Lambda\subseteq F\Lambda$ (see \cite[Proposition 2.1]{bagr13}). Applying the condition \eqref{eq:FLC} to a small compact neighborhood of the identity, it is not hard to see that the FLC condition implies uniform discreteness.

\begin{prop}
Let $\Lambda$ be an FLC point set in a second countable locally compact group. Then the discrete hull $\Omega_0(\Lambda)$ is totally disconnected.
\end{prop}

\begin{proof}
Fix a compact set $K \subseteq G$ and denote by $\cF\subseteq 2^G$ the finite collection of finite subsets of $G$ given by the FLC condition applied to $K$. For any $P\in\hull$, we have $K\cap P=hF$, for some $h\in G$ and $F\in\cF$. The sets $$A_{F,K}:=\{P\in\hull\mid \exists h\in G\ K\cap P = hF\}$$ are clearly disjoint and closed for every $F\in\mathcal{F}$. Since the collection $\{A_{F,K}\}_{F\in\cF}$ covers $\hull$, each $A_{F,K}$ is also open. To finish the proof, we argue that for any $P,Q\in\hull$ with $P\neq Q$, we can find a compact set $K\subseteq G$ and $F\in \cF$ such that $P\in A_{F,K}$ and $Q\not\in A_{F_K}$. 

Choose two such sets $P,Q\in\hull$ and without loss of generality let $x\in P\setminus Q$. Pick a compact set $K\subseteq G$ such that $K\cap P = \{e,x\}$ and $x^{-1}\not\in K$. By applying the FLC condition to $K$ we find a finite family $\cF$ of finite subsets of $G$ and two elements $g,h\in G$ such that $K\cap P=gF$ and $K\cap Q=hF'$ for some $F,F'\in\cF$. Since $e\in Q$ and $x\not\in K$, there is an element $y\in Q$ such that $K\cap Q=\{e,y\}$. Assume towards a contradiction that $F=F'$. Then, there is an equality of sets $$F= \{g^{-1},g^{-1}x\}=\{h^{-1},h^{-1}y\}.$$ Since $x\not\in Q$, then $x\neq y$ which forces $g^{-1}=h^{-1}y$ and $h^{-1}=g^{-1}x$. But then $g^{-1}=h^{-1}y=g^{-1}xy$. This implies $x=y^{-1}$, which is a contradiction since $y=x^{-1}\not\in K$ by assumption on $K$. Hence $P\in A_{F,K}$ and $Q\not\in A_{F,K}$, as desired.
\end{proof}
Note also that the groupoid $\glambda$ is ample when $\Lambda$ is FLC, since it is both étale and has a totally disconnected unit space.

\begin{example}\thlabel{ex:lattice}
Let $\Lambda$ be a lattice in a locally compact group $G$. Then $\Lambda$ is FLC since $\Lambda^{-1}\Lambda=\Lambda$ is uniformly discrete. The groupoid $\mathcal{G}(\Lambda)$ can be described the following way. Let $P\in\hull$. There exists a sequence $\{\lambda_n\}_{n\in\N}\subseteq \Lambda$ such that $P=\lim \lambda_n^{-1}\Lambda$. Since $\lambda_n^{-1}\Lambda=\Lambda$ for all $n\in\N$, we conclude that $P=\lim_n\Lambda=\Lambda$. This shows that $\hull=\{\Lambda\}$. One can further identify $\Omega(\Lambda)\setminus\{\emptyset\}$ with $G/\Lambda$ as a $G$-space. Putting everything together, one obtains:
\begin{align*}
    \glambda&=\{(x,P)\in G\times\Omega(\Lambda)\mid x^{-1}\in P \text{ and } xP,P\in\hull\}\\
    &=\{(x,g\Lambda)\in G\times G/\Lambda\mid g\Lambda=xg\Lambda = \Lambda\}\\
    &=\{(x,\Lambda)\in G\times\{\Lambda\}\mid x\in\Lambda\}=\Lambda.
\end{align*}
\end{example}
Hence, in the case where $\Lambda$ is a lattice, the groupoid construction $\glambda$ recovers $\Lambda$. For general point sets, the $G$-space $\Omega(\Lambda)\setminus\{\emptyset\}$ appears to be a good replacement of $G/\Lambda$. Recall that a lattice $\Lambda$ in a group $G$ is a discrete subgroup such that the $G$-space $G/\Gamma$ has a finite invariant measure. By relaxing the subgroup condition and replacing $G/\Lambda$ by $\Omega(\Lambda)$, one obtains the notion of strong approximate lattice (see \cite[Definition 1.3]{BjHa18}).

\begin{example}\thlabel{ex:Meyer-set}
Let $G$ and $H$ be second countable locally compact groups and call $\pi_G$ and $\pi_H$ the projections from $G\times H$ onto $G$ and $H$ respectively. Let $\Gamma\subseteq G\times H$ be a lattice such that $\pi_G$ is injective when restricted to $\Gamma$ and such that $\pi_H(\Gamma)$ is dense in $H$. Let $W$ be a symmetric compact neighborhood of the identity in $H$. The set $$\Lambda=\Lambda(\Gamma,W) = \pi_G((G\times W)\cap\Gamma)\subseteq G$$ is called a \textit{model set}. Note that 
$$\Lambda^{-1}\Lambda\subseteq\pi_G((G\times W^{-1}W)\cap\Gamma),$$ and the latter is uniformly discrete. Hence $\Lambda$ is FLC.
\end{example}

\section{Inner amenability of groupoids associated with point sets}
In this section, we introduce inner amenability and prove the main result of this note (see \thref{thm:inner-amenability-FLC}), namely that groupoids associated with FLC point sets in second countable locally compact groups are inner amenable.
\subsection{Inner amenable groupoids} We will summarize \cite[Section 5]{an21} which presents the current state of the art concerning inner amenability of locally compact second countable groupoids. The original definition and idea of inner amenability come from the following equivalence for locally compact groups in \cite[Proposition 1]{lori87}.  A locally compact $G$ admits a conjugation invariant mean on $L^\infty(G)$ if and only if there is a net $\xi_i$ of positive functions on $C_0(G)$ in the unit ball of $L^2(G)$ such that the matrix coefficients $\langle \xi_i,\lambda_s\rho_s\xi_i\rangle$ converge uniformly on compact subsets of $G$ to $1$. A locally compact group satisfying this property is called inner amenable. Denoting by $f_i$ the matrix coefficients 
    $$f_i:G\times G\to \C:(s,t)\mapsto\langle\xi_i,\lambda_s\rho_t\xi_i\rangle,$$
one sees that $f_i$ is a net of continuous, positive definite functions converging uniformly to $1$ on the diagonal of $G\times G$ such that $\supp(f_i)\cap(G\times K)$ and $\supp(f_i)\cap (K\times G)$ is compact for every $i$ and every $K\subseteq G$ compact. In fact (see \cite[Theorem 3.5]{CrTa17}) the existence of such a net of continuous functions is equivalent to the existence of a conjugation invariant mean on $L^\infty(G)$. Note that the existence of Dirac measures at the identity witnesses inner amenability for discrete groups. This definition of inner amenability in terms of positive type functions has been generalized to the framework of groupoids. The study of inner amenability for groupoids associated with point sets in locally compact groups is the content of this piece.

We now introduce the appropriate notion of inner amenability for groupoids (see \cite[Section 5]{an21}). For a groupoid $\G$, a function $f$ on $\G\times\G$ is \textit{properly supported} if for all compacts $K\subseteq \G$, the sets $\supp(f)\cap (K\times\G)$ and $\supp(f)\cap (\G\times K)$ are compact. A function $f$ is said to be of \textit{positive type} if for all units $x,y\in \G^{(0)}$, and $\gamma_1,\ldots,\gamma_n\in \G^x$, $\eta_1,\ldots,\eta_n\in \G^y$, the matrix $(f(\gamma_i^{-1}\gamma_i,\eta_j^{-1}\eta_j))_{i,j}$ is positive definite. 

\begin{defn}\cite[Definition 5.3]{an21}
A groupoid $\G$ is inner amenable if for all compact $K\subseteq \G$ and $\epsilon>0$, one can find a positive type, continuous, properly supported function $f$ on $\G\times\G$ with $\mid f(\gamma,\gamma)-1\mid<\epsilon$ for all $\gamma\in K$. 
\end{defn}
The existence of such functions for amenable locally compact groupoids is classical, see \cite[Chapter 2]{AnRe00}. So amenable groupoids are inner amenable. When the groupoid $\G$ is a discrete group, and $(\lambda\times\rho,\ell^2(\G))$ denotes the product of left and right regular representations on $\ell^2(\G)$ then the matrix coefficient associated with $\delta_e\in L^2(\G)$ in the representation $\lambda\times\rho$ witnesses inner amenability of $\G$. Recall that a continuous map between groupoids which respects the source, target and multiplication is called a \textit{ groupoid morphism}. A groupoid morphism $\rho:\mathcal{H}\to\G$ is \textit{locally proper} if the map 
\begin{align*}
    \psi = \rho\times s\times r :\G&\to \mathcal{H}\times\Gnought\times\Gnought\\
    g&\mapsto ((\rho(g),s(g),r(g))),
\end{align*}
is proper. An example of such a map is the inclusion map of a closed subgroupoid (see \cite[Example 4.19(a)]{an21}). The main tool to find examples of inner amenable groupoids is the following proposition.
\begin{prop}\cite[Proposition 5.6]{an21}\thlabel{prop:locally-proper}
 Let $\G$ and $\mathcal{H}$ be locally compact groupoids. Assume there exists a locally proper groupoid morphism $\rho:\mathcal{H}\to\G$. If $\G$ is inner amenable, then so is $\mathcal{H}$. 
\end{prop}
Using the inclusion map of closed subgroupoids, the fact that closed subgroupoids of inner amenable groupoids are inner amenable can be easily deduced. Another consequence of \thref{prop:locally-proper} is that étale transformation groupoids are inner amenable. Indeed, if a discrete group $\Gamma$ acts on a space $X$, then the projection map $X\rtimes \Gamma\to\Gamma$ is locally proper and $\Gamma$ is inner amenable because it is a discrete group, hence the inner amenability of $X\rtimes\Gamma$. After establishing this list of properties, Anantharaman-Delaroche asks whether all étale groupoids are inner amenable (see \cite[Problem, Section 5]{an21}). The groupoid attached to a point set in a locally compact group is an example of an inner amenable étale groupoid, which is not a transformation groupoid known to be inner amenable, unless the point set is a lattice or the group is inner amenable.

\subsection{Inner amenability for FLC point sets}
In this section, we prove that groupoids attached to point sets with finite local complexity in second countable locally compact groups are inner amenable. In what follows, we fix a locally compact second countable topological group $G$ and an FLC point set $\Lambda\subseteq G$. Recall that $\Lambda$ is in particular uniformly discrete and we will denote by $U\subseteq G$ an open set with respect to which $\Lambda$ is $U$-discrete. 

The following lemma will be central to the construction of continuous functions from groupoids associated with FLC point sets.
\begin{lem}\thlabel{lem:FLC-discretehull}    
Let $\Lambda\subseteq G$ be an FLC point set in a locally compact group. Then the set $X=\bigcup_{P\in\hull}P$ is discrete.
\end{lem}
\begin{proof}
It is enough to show that for every compact set $K\subseteq G$, $K\cap X$ is finite. Let $K$ be such a set. We may assume that the identity element $e$ of $G$ is in $K$, otherwise we may replace $K$ by a bigger compact set $K'$ in $G$ containing $K$ and $e$. Let $\cF\subseteq\cP(G)$ be a finite collection of finite subsets of $G$ given by the FLC condition applied to $K$. In other words, for every $P\in \hull$, there exists $h\in G$ and $F\in\cF$ such that $P\cap K=hF$. We claim that $$X\cap K\subseteq\bigcup_{\substack{F\in\cF,\\ h\in F^{-1}}}hF,$$ which is finite. Let $x\in X\cap K$. Then there exists $P\in\hull$ such that $x\in P\cap K$. Let $h\in G$ and $F\in \cF$ such that $P\cap K=hF$. Since $e\in P\cap K$, one has $h^{-1}\in F$, hence $h\in F^{-1}$. As a consequence, $$x\in hF\subseteq F^{-1}F\subseteq\bigcup_{\substack{F\in\cF,\\ h\in F^{-1}}}hF,$$ as desired. 
\end{proof}

We state and prove the main result of this note now.
\begin{thm}\thlabel{thm:inner-amenability-FLC}
Let $\Lambda\subseteq G$ be an FLC point set in a locally compact second countable group $G$. Then the groupoid $\G(\Lambda)$ is inner amenable.
\end{thm}

\begin{proof}
Let $\delta_e:G\to\{0,1\}$ denote the Dirac function at $e\in G$ defined by $\delta_e(g)=1$ if and only if $g=e$, and $0$ otherwise. We claim that the function 
\begin{align*}
    \psi:\G(\Lambda)\times \G(\Lambda) &\to \C\\
    ((x,P),(y,Q))&\mapsto \delta_e(x^{-1}y)
\end{align*}
is continuous, positive type, properly supported and satisfies $\psi(\gamma,\gamma)=1$ for any $\gamma \in\glambda$. In particular, $\psi$ witnesses the inner amenability of $\glambda$.

We first check that $\psi$ has proper support, let $K\subseteq\G(\Lambda)$ be compact. Since every compact set in $\glambda$ is a closed subset of $\{(x,P)\in\glambda\mid x\in C\}$ for some $C\subseteq G$ compact, we may assume that $K$ is of this form.
Then $$\supp(\psi)\cap(K\times\glambda)\subseteq ((C\times \hull)\cap \glambda)\times((C\times \hull)\cap \glambda),$$ which is compact. A similar argument also shows that $\supp(\psi)\cap(\glambda\times K)$ is compact.

Further, observe that $$\psi(\gamma,\gamma)=\delta_e(x^{-1}x)=\delta_e(e)=1,$$ for any $\gamma=(x,P)\in\glambda$.

Now, we show that $\psi$ is positive type. Let $P,Q\in \Omega_0(\Lambda)=\G(\Lambda)^{(0)}$ be units, $\gamma_1,\ldots,\gamma_n\in \glambda^P$ and $\eta_1,\ldots,\eta_n\in \glambda^Q$. Say $\gamma_1=(x_1,x_1^{-1}P),\ldots,\gamma_n=(x_n,x_n^{-1}P)$ and $\eta_1=(y_1,y_1^{-1}Q),\ldots,\eta_n=(y_n,y_n^{-1}Q)$ for some $x_1,\ldots,x_n\in P$ and $y_1,\ldots,y_n\in Q$. Let $M$ be the following matrix. $$(M_{ij})_{ij}=(\psi(\gamma_i^{-1}\gamma_j,\eta_i^{-1}\eta_j))_{i,j}=(\delta_e({x_j^{-1}x_i}{y_i^{-1}y_j}))_{i,j}\in M_n(\{0,1\}).$$
The matrix $M$ is symmetric since $M_{ij}=1$ if and only if ${x_j^{-1}x_i}={y_j^{-1}y_i}$ which is equivalent to ${x_i^{-1}x_j}={y_i^{-1}y_j}$ which is in turn equivalent to $M_{ji}=1$. Note also that $M_{ik}=1$ if $M_{ij}=M_{jk}=1$. Indeed, the latter condition implies that both ${x_j^{-1}x_i}={y_j^{-1}y_i}$ and ${x_k^{-1}x_j}={y_k^{-1}y_j}$ hold. Multiplying the two equations, we obtain ${x_k^{-1}x_i}={y_k^{-1}y_i}$ which implies $M_{ik}=1$. As a consequence, up to reordering the columns and rows of $M$, the matrix $M$ is of the form $M=diag(M_1,\ldots,M_k)$, where $M_0$ only entries are all $0$, $M_k$ only entries are all $0$ for $k=1,\ldots n$ and $l_0+l_1+\cdots+l_k=n$. In other words, the matrix $M$ is a sum of (multiples of) orthogonal projections, hence it is positive.

Lastly, we check the continuity of $\psi$. Denote by $X$ the space $\bigcup_{P\in\hull}P$. Denoting by $\pi:\glambda\to X$ the projection on the first coordinate, we get that $\psi=\overline{\psi}\circ (\pi\times\pi)$, for some map $\overline{\psi}:X\times X\to \C$. Since $X$ is discrete by \thref{lem:FLC-discretehull}, the map $\overline{\psi}$ is continuous. Hence $\psi$ is continuous.

\end{proof}

The groupoids $\glambda$ and $G\ltimes\Omega^\times(\Lambda)$ are equivalent (see \cite[Section 3]{enstad20}). Hence, we obtain the following consequence as an application of \thref{thm:inner-amenability-FLC} together with \cite[Proposition 8.11]{an21}.
\begin{cor}\thlabel{cor:inner-amenability}
Let $\Lambda$ be an FLC point set in a locally compact second countable group $G$. The following are equivalent.
\begin{enumerate}[label=(\roman*)]
    \item The groupoid $\G(\Lambda)$ is (topologically) amenable;
    \item the action $G\curvearrowright\Omega(\Lambda)$ is amenable;
\end{enumerate}
\end{cor}
One of the applications of \thref{thm:inner-amenability-FLC} is establishing that the groupoids associated with approximate lattices are inner amenable, even in non-amenable groups. In particular, the groupoids associated to model sets (see \thref{ex:Meyer-set}) are inner amenable. An interesting class of examples of model sets arise from lattices in arithmetic groups.
\begin{example}\label{ex:arithmetic}
Consider the embedding 
\begin{align*}
    \iota:\Z[\sqrt{2}]&\to \R\times\R\\
    a+b\sqrt{2} &\mapsto (a+b\sqrt{2},a-\sqrt{2}).
\end{align*}
Using this embedding, we view $SL_n(\Z[\sqrt{2}])$ as an irreducible lattice in $SL_n(\R)\times SL_n(\R)$. If $W$ is any symmetric compact neighborhood of the identity in $SL_n(\R)$ and $\pi_1:SL_n(\R)\times SL_n(\R)\to SL_n(\R)$ denotes the projection on the first coordinate, then the model set $$\Lambda=\pi_1(SL_n(\Z[\sqrt{2}])\cap (SL_n(\R)\times W))$$ is FLC (see Example \ref{ex:Meyer-set}). As a consequence of \thref{thm:inner-amenability-FLC}, the groupoid $\glambda$ is inner amenable. Observe also that there is an $SL_n(\R)$-equivariant inclusion $\iota:\Linfty(SL_n(\R))\to\Linfty(SL_n(\R)\times \Omega(\Lambda))$ induced by the projection $SL_n(\R)\times \Omega(\Lambda)\to SL_n(\R)$. If $\glambda$ was amenable, then by \thref{cor:inner-amenability} there would exists a $G$-invariant mean on $\Linfty(SL_n(\R)\times \Omega(\Lambda))$. Precomposing that mean with the inclusion $\iota$ would give an $SL_n(\R)$-invariant (for the left translation action) mean on $\Linfty(SL_n(\R))$ which would imply the amenability of $SL_n(\R)$. Hence the groupoid $\glambda$ is not amenable.
\end{example}
One can replace the number $2$ by any square free positive integer equal to $2$ or $3$ modulo $4$ in the Example \ref{ex:arithmetic}. More generally, one can consider any totally real Galois extension $F$ of degree $n$ over $\mathbb{Q}$. The analogue of Example \ref{ex:arithmetic} holds provided $\Z[\sqrt{2}]$ is replaced by the ring of integers $\mathcal{O}_F$ of $F$ and the embedding $\iota$ is replaced by 
\begin{align*}
    \iota:\mathcal{O}_F&\to\R^n\\
 z&\mapsto (\sigma_1(z),\ldots,\sigma_n(z)),
\end{align*}
where $\sigma_1,\ldots,\sigma_n$ are real embeddings of $F$ over $\mathbb{Q}$.

\printbibliography
\end{document}